\documentclass[12pt]{amsart}
\usepackage{amssymb,amsmath,amstext}

\theoremstyle{plain}
 
\newtheorem{theorem}{Theorem}[section]

\newtheorem{proposition}[theorem]{Proposition}
\newtheorem{lemma}[theorem]{Lemma}

\theoremstyle{definition}

\theoremstyle{remark}

\newcommand{\I}{\mathcal{I}}

\begin{document}

\title{The $(p,q)$ property in families of $d$-intervals and $d$-trees}

\author{Shira Zerbib}
\address{Department of Mathematics,
University of Michigan, Ann Arbor} \email{zerbib@umich.edu}


\begin{abstract}
Given integers $p\ge q>1$, a family of sets satisfies the $(p,q)$ property if among any $p$ members of it some $q$ intersect. We prove that for any fixed integer constants $p\ge q>1$, a family of $d$-intervals satisfying the $(p,q)$ property can be pierced by $O(d^{\frac{q}{q-1}})$ points, with constants depending only on $p$ and $q$. This extends results of Tardos, Kaiser and Alon for the case $q=2$, and of Kaiser and Rabinovich for the case $p=q=\lceil log_2(d+2) \rceil$. We further show that similar bounds hold in families of subgraphs of a tree or a graph of bounded tree-width, each consisting of at most $d$ connected components, extending results of Alon for the case $q=2$. Finally, we prove an upper bound of $O(d^{\frac{1}{p-1}})$ on the fractional piercing number in families of $d$-intervals satisfying the $(p,p)$ property, and show that this bound is asymptotically sharp.
\end{abstract}

\maketitle


\section{Introduction}
A {\em homogenous $d$-interval} is a union of $d$ disjoint closed intervals on the real line. A {\em separated $d$-interval} is a union of $d$
non-empty closed intervals, one on each of $d$ fixed pairwise disjoint segments on the real line. It is easy to see that a family of separated $d$-intervals is, in particular, a family of homogeneous $d$-intervals. Thus, when discussing families of homogeneous $d$-intervals, we will sometimes omit the word `homogeneous'.  

Given a finite family of (homogeneous,  separated) $d$-intervals, we are interested in properties of the hypergraph $H$ whose vertex set
is the real line, and whose edges correspond to the $d$-intervals in the family, viewed as sets of points. We call $H$ a \emph{hypergraph of (separated, homogeneous) $d$-intervals}.  

A {\em matching} in a hypergraph $H$ with vertex set $V$ and
edge set $E$ is a subset of disjoint edges in $E$. A {\em cover} is a subset of
$V$ intersecting all edges. The \emph{matching number} $\nu(H)$ is the maximal size of a matching in $H$, and
the \emph{covering number} $\tau(H)$ is the minimal size of a
cover in $H$. The fractional relaxations of $\nu$ and $\tau$ are
denoted, as usual, by $\nu^*$ and $\tau^*$, respectively. By linear programming duality, $\nu\le \nu^*=\tau^*\le \tau$.

An old result of Gallai is that if $H$ is a hypergraph of intervals, then $\tau(H)=\nu(H)$. Tardos \cite{tardos} and Kaiser \cite{kaiser} used topological methods to prove that a hypergraph $H$ of homogenous $d$-intervals has $\frac{\tau(H)}{\nu(H)} \leq d^2-d+1$, and this bound improves to $\frac{\tau(H)}{\nu(H)} \leq d^2-d$ in the separated case. Another approach was taken by Alon \cite{alon} who proved a slightly weaker
upper bound in the homogenous case, namely $\frac{\tau}{\nu} \leq
2d^2$,
by breaking his bound into two results: $\tau^* \le 2d\nu$,
and $\tau \le d\tau^*$. In \cite{AKZ} examples for the sharpness of these two fractional bounds as well as an alternative proof of the latter are given.
Matou\v{s}ek \cite{matousek} showed that the quadratic bound in $d$ on the ratio $\tau/\nu$ in hypergraphs of $d$-intervals is not far from
being optimal: there are examples of intersecting hypergraphs of
$d$-intervals in which $\tau=\frac\tau\nu = \Omega(\frac{d^2}{\log
  d})$.

A hypergraph $H$ is said to satisfy the {\em $(p,q)$ property} if among every $p$ edges in $H$ some $q$ have a non-empty intersection.
In this terminology Gallai's theorem asserts that if an interval hypergraph $H$ satisfies the $(p,2)$ property for some fixed integer constant $p$, then $\tau(H)=p-1$. This result generalizes to all pairs of integer constants $p\ge q>1$: by a theorem of Hadwiger and Debrunner from 1957 \cite{HD}, if an interval hypergraph $H$ satisfies the $(p,q)$ property, then $\tau(H) \le p-q+1$. 

Tardos' and Kaiser's results are that in any hypergraph of $d$-intervals $H$ satisfying the $(p,2)$ property for some fixed integer constant $p$, $\tau(H)\le (p-1)(d^2-d+1).$
Kaiser and Rabinovich \cite{KR} proved that if a hypergraph $H$ of separated $d$-intervals satisfies the $(p,p)$ property for $p=\lceil \log_2(d + 2) \rceil$, then $\tau(H) \le d$. Their proof does not apply in the non-separated case. 
Bj\"orner, Matou\v{s}ek and Ziegler \cite{BMZ} then raised the question whether the $\lceil \log_2(d + 2) \rceil$ factor can be replaced by some constant bound on $p$ independent of $d$. Matou\v{s}ek's example shows that such a constant must be greater than $2$. For a survey on the $(p,q)$ property in families of convex sets and $d$-intervals see \cite{Eckhoff}.

Viewed as a discrete object, a family of $d$-intervals is a collection of subgraphs of a path $G$, each consisting of at most $d$ connected components. 
Alon \cite{alon1} extended this setting to collections of subgraphs of a tree $G$. Let $G$ be a connected graph and let $H=H(G)$ be a collection of subgraphs of $G$, each consisting of at most $d$ connected components. If $G$ is a tree, we call $H(G)$ a {\em hypergraph of $d$-trees}. As before, $H(G)$ satisfies the {\em $(p,q)$ property} if in every $p$ edges of it there exists some $q$ edges sharing a common vertex. 
Alon proved that if a hypergraph of $d$-trees $H$ satisfies the $(p,2)$ property for some fixed integer $p>1$, then $\tau(H)\le 2(p-1)d^2$. In the case $G$ is an arbitrary graph of bounded tree-width $k$ (see definition in Section \ref{sec:dtw}) satisfying the $(p,2)$ property, 
he showed $\tau(H(G))\le 2(k+1)(p-1)d^2$.

In this paper we extend the above mentioned results, by proving upper bounds on the piercing numbers in families of $d$-intervals or $d$-trees satisfying the $(p,q)$ property, for any two fixed integer constants $p\ge q >1$. 
In all our results we are interested in the asymptotic order of the piercing numbers as $d$ goes to infinity, and make no attempt to optimize the involved constants depending on fixed $p,q$.

For families of separated $d$-intervals we prove the following.

\begin{theorem}\label{dpps}
If a hypergraph $H$ of separated $d$-intervals satisfies the $(p,p)$ property for some fixed integer constant $p>1$, then $\tau(H) \le d^{\frac{p}{p-1}}$.
\end{theorem}

The proof of Theorem \ref{dpps} uses
a topological theorem due to Komiya \cite{komiya} together with a counting argument generalizing \cite{KR}. In the case $p=\lceil \log_2(d + 2) \rceil$ it produces $\tau(H) \le d$, as is obtained in \cite{KR}.

For families of homogenous $d$-intervals we have: 

\begin{theorem}\label{dpp}
Let $H$ be a hypergraph of homogeneous $d$-intervals satisfying the $(p,p)$ property, for some fixed integer constant $p>1$. Then  
$\tau^*(H) < p^{\frac{1}{p-1}}d^{\frac{1}{p-1}}+1 $, $\tau(H) < p^{\frac{1}{p-1}}d^{\frac{p}{p-1}}+d $, and the bound on $\tau^*(H)$ is asymptotically sharp: for every integer $p>1$, there exists a hypergraph of homogeneous $d$-intervals $H$ satisfying the $(p,p)$ property with $\tau^*(H) = \Omega(d^{\frac{1}{p-1}})$.
\end{theorem}  

\begin{theorem}\label{dpq}
If a hypergraph $H$ of homogeneous $d$-intervals satisfies the $(p,q)$ property for some fixed integer constants $p\ge q>1$, then $$\tau(H)\le \max\Big\{\frac{2^{\frac{1}{q-1}}(ep)^\frac{q}{q-1}}{q} d^{\frac{q}{q-1}} + d,~ 2p^2d\Big\}.$$
\end{theorem}

Our next two theorems extend the above to families of $d$-trees. Although Theorems \ref{dpp} and \ref{dpq}  are special cases, their proofs are simpler and have geometric nature, and thus we consider them separately.

For families of $d$-trees we prove:

\begin{theorem}\label{treedpp}
Let $G$ be a tree, and let $H=H(G)$ be a hypergraph of $d$-trees satisfying the $(p,p)$ property for some fixed integer constant $p>1$. Then $\tau(H)< p^{\frac{1}{p-1}}d^{\frac{p}{p-1}} + d.$
\end{theorem}

\begin{theorem}\label{treedpq}
Let $G$ be a tree, and let $H=H(G)$ be a hypergraph of $d$-trees satisfying the $(p,q)$ property for some fixed integer constants $p\ge q>1$. Then $$\tau(H)\le \max\Big\{\frac{2^{\frac{1}{q-1}}(ep)^\frac{q}{q-1}}{q} d^{\frac{q}{q-1}} + d,~ 2p^2d\Big\}.$$
\end{theorem}

The 
proofs of Theorems \ref{dpps}, \ref{dpp} and \ref{dpq} are given in Sections \ref{sec:dpps}, \ref{sec:dpp} and \ref{sec:dpq}, respectively.
In Section \ref{sec:lemma} we prove a lemma concerning collections of subtrees of a tree $G$, 
and the proofs of Theorems \ref{treedpp} and \ref{treedpq} are given in Section \ref{sec:dtree}. In section \ref{sec:dtw} we further extend our results to collections of subgraphs of an arbitrary graph $G$ of bounded tree-width. 

\begin{theorem}\label{tw}
Let $G$ be a graph of tree-width at most $k$, and let $H$ be a collection of subgraphs of $G$, each consisting of at most $d$ connected components. If $H$ satisfies the $(p,q)$ property for some fixed integer constants $p\ge q>1$, then $$\tau(H)\le (k+1)\cdot\max\Big\{\frac{2^{\frac{1}{q-1}}(ep)^\frac{q}{q-1}}{q} d^{\frac{q}{q-1}} + d,~ 2p^2d\Big\}.$$
\end{theorem}

In our proofs we apply some of the methods used in \cite{alon, alon1, alon2, AKZ} together with additional counting arguments. In most cases we first prove that the fractional covering number $\tau^*(H)$ cannot exceed a certain bound, and then use Alon's upper bound on the ratio $\tau(H)/\tau^*(H)$:

\begin{theorem}[Alon, \cite{alon, alon1}]\label{alon}
Let $G$ be a tree, and let $H=H(G)$ be a hypergraph of $d$-trees. Then $\tau(H) \le d\tau^*(H)$. In particular, if $H$ is a hypergraph of $d$-intervals then $\tau(H) \le d\tau^*(H)$.
\end{theorem}

\section{The $(p,p)$ property in families of separated $d$-intervals}\label{sec:dpps}

For a positive integer $d$, we denote by $[d]$ the set of integers $\{1,\dots,d\}$. 
A $d$-uniform hypergraph $D=(V,E)$ is {\em $d$-partite} if its vertex set $V$ is a disjoint union of $d$ sets $V=V_i \cup\dots \cup V_d$, and every edge $e\in E$ has $|e\cap V_i|=1$ for all $i\in [d]$. 
For every subset $S\subset V$, denote by $\chi_S \in \{0,1\}^{V}$ the characteristic vector of $S$, namely, $\chi_S(v)$ is $1$ if $v\in S$, and $0$ otherwise.  
A hypergraph $D=(V,E)$ is said to be {\em balanced} if there exist positive weights $\alpha_e$, $e\in E$, such that $\sum_{e\in E} \alpha_e \chi_e = \chi_V$. 

For the proof of Theorem \ref{dpps} we will need the following lemma,  generalizing a lemma proved in \cite{KR}.   
\begin{lemma}\label{lemmapp}
Let $D=(V,E)$ be a balanced $d$-partite hypergraph on vertex set $V=V_i \cup\dots \cup V_d$, where $|V_i|=t$ for all $i$.  
If every $\ell+1$ edges of $H$ have at least $m$ vertices in common, then every $\ell$ edges of $H$ have at least $tm+1$ vertices in common.
\end{lemma}
\begin{proof}
Assume to the contrary that there exists a set $A$ of $\ell$ edges in $E$ such that $|\bigcap A| \le tm$. Write $C=\bigcap A$, and note that $|C\cap V_i|\le 1$ for every $i\in [d]$. Let $I\subset [d]$ be the set of indices $i$ such that $|C\cap V_i| =1$. Write $\bar{C} = \big(\bigcup_{i\in I} V_i\big) \setminus C$. Since  $|V_i|=t$ for all $i$, we have $|\bar{C}|=(t-1)|C|$. 

By the conditions of the lemma,  
for every edge $e \notin A$, the set $A\cup \{e\}$ have at least $m$ vertices in common, and these vertices must be in $C$. Since $D$ is balanced this implies   
$$|C| \ge |C|\sum_{e \in A} \alpha_e + m\sum_{e \notin A} \alpha_e > m\sum_{e \notin A} \alpha_e. $$ Therefore, from $|C| \le tm$ we get $$|\bar{C}| \le (|C|-m)\sum_{e \notin A} \alpha_e \le (t-1)m\sum_{e \notin A} \alpha_e < (t-1)|C|,$$ a contradiction.
\end{proof}

We will further need a topological theorem due to Komiya \cite{komiya}. The application of Komiya's theorem in the proof of Theorem \ref{dpps} will be similar to its application in the proof of Theorem 6.3 in \cite{AKZ}.  
\begin{theorem}[Komiya, \cite{komiya}] \label{komiya}
Let $P$ be a polytope, and for 
every face $F$ of $P$ choose a point  $q(F)\in F$ and an open set $B_F \subseteq P$. If $G\subseteq \bigcup_{F \subseteq G} B_F$
for every face $G$ of $P$, 
then there exists a collection $Q$ of faces of $P$ satisfying $q(P) \in
\text{conv}\{q(F) \mid F \in Q\}$ and $\bigcap_{F\in Q} B_{F}\neq
\emptyset$.
\end{theorem}

We are now ready to prove Theorem \ref{dpps}. 
\medskip

\noindent {\em Proof of Theorem \ref{dpps}.}
Since $H$ is finite, we may
assume that the vertex set of $H$ is the union of $d$ disjoint copies of open unit segment $(0,1)$. Write $k= \lfloor d^{\frac{1}{p-1}}\rfloor$.
We apply Komiya's theorem to the polytope
$P=\Delta_k \times \Delta_k \times \ldots \times \Delta_k$, the
$d$-fold product of the $k$-dimensional simplex $\Delta_k$ by
itself. A point $\vec{x} \in P$ has the form:
$$\vec{x}=((x^1_1, x^1_2, \ldots ,x^1_{k+1}), (x^2_1, x^2_2,
\ldots ,x^2_{k+1}), \ldots ,(x^d_1, x^d_2, \ldots ,x^d_{k+1})),$$
where $x^j_i \ge 0$ and $\sum_{i=1}^{k+1} x^j_i =1$ for every $1\le j\le d$.  
For $\vec{x} \in P$ and $(i,j)\in [k+1]\times[d]$ let
$p_{\vec{x}}(i,j)=\sum_{t \le i}x_t^j$. 
A $0$-dimensional face of $P$ is a point $\vec{y}\in P$ consisting of only $0$ and $1$ components. 

For the application of Theorem \ref{komiya}, define sets $B_F$ and points $q(F)$ for every face $F$ of $P$ as follows.  
For a 0-dimensional face $\vec{y}\in P$ let  $B_{\vec{y}}$ be the set
of all points $\vec{x} \in P$ for which there exists a separated $d$-interval $h \in H$
satisfying $h^j \subseteq (p_{\vec{x}}(i^j(\vec{y})-1,j), p_{\vec{x}}(i^j(\vec{y}),j))$ for all $j$, where $h^j$ is the $j$-th interval component of  $h$, and $i^j(\vec{y})\in [k+1]$ is the unique index $i$ such that $y_{i}^j=1$. For all other faces $F$ of $P$ let 
$B_F=\emptyset$. Since every $d$-interval in $H$ consists of closed interval components, the sets $B_F$ are open.  
For every face $F$ of $P$ choose the point $q(F)$ to be the barycenter of $F$.

Assume to the contrary that $\tau(H)>d^{\frac{p}{p-1}}$. Then for every $\vec{x}\in P$, the set of points
$\{p_{\vec{x}}(i,j) \mid 1 \le i \le k, 1 \le j \le d\} \subset V(H)$ is not a cover for
$H$. Hence $\bigcup B_{\vec{y}}=P$, where the union is over all
0-dimensional faces $\vec{y}$ of $P$. Moreover, for every face $F = \text{conv}(T)$, where $T$ is a set of
0-dimensional faces in $P$, we have 
 $F \subseteq \bigcup \{B_{\vec{y}} \mid \vec{y} \in T\}$, because an empty segment cannot contain a non-empty interval component of a $d$-interval in $H$. Thus by Theorem \ref{komiya}, there exists a set $Q$ of 0-dimensional faces of $P$, such that $q(P) \in \text{conv}\{q(\vec{y}) \mid \vec{y} \in Q\}$ and $\bigcap_{\vec{y} \in Q}B_{\vec{y}} \neq \emptyset$.

Let $D$ be the $d$-partite hypergraph on vertex set $V=\bigcup_{j=1}^{d} V^j$, where each $V^j$ is a distinct copy of $[k+1]$, and with edge set $E=\{e_{\vec{y}}\mid \vec{y} \in Q\}$, where $e_{\vec{y}}$ is the subset of $V$ satisfying $\chi_{e_{\vec{y}}}= \vec{y}$. 
Then the fact that $q(F) \in \text{conv}\{q(\vec{y}) \mid \vec{y} \in Q\}$
 means that $D$ is balanced. Moreover, since $H$ satisfies the $(p,p)$ property, $D$ must satisfy it too, that is, every $p$ edges of $D$ intersect at (at least) one vertex.
Applying Lemma \ref{lemmapp} $p-1$ times, we obtain that every edge of $D$ intersects itself at more than $(k+1)^{p-1} > d^{\frac{p-1}{p-1}}=d$ vertices. But this contradicts the fact that every edge of $D$ contains exactly $d$ vertices.  
\qed

\section{The $(p,p)$ property in families of homogeneous $d$-intervals}\label{sec:dpp}

Let $H$ be a hypergraph of homogeneous $d$-intervals. By performing small perturbations on the $d$-intervals of $H$,
we may assume that no two endpoints of intervals in edges of $H$ coincide.
We begin this section by proving a simple lemma which will play a crucial role in the sequel.

\begin{lemma}\label{firstlemma}
If a finite family $\I$ of $k$ closed intervals on a line intersects, then there exist at least two pairs $(x,\I\setminus \{I\})$, where $x$ is
an endpoint of the interval $I \in \I$  and $x$ lies in every interval $J\in \I$.
\end{lemma}
\begin{proof}
Let $K = \bigcap \I$. Then $K$ is a non-empty closed interval $[x_1,x_2]$, where $x_1, x_2$ are endpoints of (not necessarily distinct) intervals $I_1,I_2 \in \I$, respectively, and in particular, $x_1\neq x_2$. Thus $(x_1,\I\setminus \{I_1\})$ and $(x_2,\I\setminus \{I_2\})$ are the two pairs in the lemma.
\end{proof}

{\em Proof of Theorem \ref{dpp}}.
Let $f:H \to \mathbb{Q}^+$ be a maximal fractional matching. 
Let $r$ be a common
denominator for the values of $f$. By duplicating edges and removing edges $h$ with
$f(h) = 0$, we obtain a hypergraph $H'$ with the same relevant parameters as $H$, and
with maximal degree $r$.
Thus 
without loss of generality we may assume that $f(h)=\frac{1}{r}$ for all $h
\in H$, where $r$ is the maximal size of a subset of edges in $H$ intersecting in a single point. Note that the removal and duplication of edges does not invalidate the $(p,p)$ property of $H$. In particular, we have $r\ge p$. Letting $n$ be the number of edges in $H$ we have by LP duality that $\tau^*(H)=\nu^*(H)=\frac{n}{r}$, and thus we need to show that
$\frac{n}{r} < p^{\frac{1}{p-1}}d^{\frac{1}{p-1}} + 1$.

Since $H$ satisfies the $(p,p)$ property, it contains $\binom{n}{p}$ subsets of $p$ edges, each of which intersects at a vertex of $H$. By Lemma \ref{firstlemma}, each such subset contributes at least two different pairs $(x,\{h_1,\dots,h_{p-1}\})$, where $x$ is
an endpoint of an edge $h$ of $H$, the edges $h_1,\dots,h_{p-1}$ of $H$ are distinct and all different from $h$, and $x$ lies in every $h_i$, $1\le i\le p-1$.

Since there are altogether at most $2dn$ possible choices for $x$, there is such a point $x$ that belongs to at least $$X:=\frac{2}{2dn}\binom{n}{p} = \frac{\prod_{i=1}^{p-1}(n-i)}{d\cdot p!} \ge \frac{(n-p+1)^{p-1}}{d\cdot p!}$$ different pairs $(x,\{h_1,\dots,h_{p-1}\})$. 
If the number of edges $h\in H$ containing $x$ is $Z$,
then, by the above, $\binom{Z}{p-1} \ge  X$, from which it follows that $$ Z\ge \Big((p-1)!\cdot X\Big)^{\frac{1}{p-1}} \ge \frac{n - p+1}{(pd)^{\frac{1}{p-1}}}.$$ Thus, in particular,  $$r \ge \frac{n - p+1}{(pd)^{\frac{1}{p-1}}},$$ implying 
$$\frac{n - p+1}{r(pd)^{\frac{1}{p-1}}}\le 1.$$ 
Since $r\ge p$ this entails $$\tau^*(H)=\frac{n}{r} \le (pd)^{\frac{1}{p-1}} + \frac{p-1}{r}<(pd)^{\frac{1}{p-1}} + 1.$$
 
Now, by Theorem \ref{alon}
we have $\tau(H) \le d\tau^*(H)$, and therefore,
$$\tau(H) < p^{\frac{1}{p-1}}d^{\frac{p}{p-1}} + d.$$

To see that the bound on $\tau^*(H)$ is asymptotically sharp,
let $q$ be a prime power and let $P=\mathbb{P}^k(\mathbb{F}_q)$ denote the $k$-dimensional projective space over the field $\mathbb{F}_q$ of $q$ elements. Let $H=(V,E)$ be the hypergraph with vertex set $V=P$ whose edge set $E$ is the family of $(k-1)$-dimensional projective subspaces of $P$. Clearly, $H$ satisfies the $(k,k)$ property.

\begin{proposition}
$\tau^*(H) = q+\frac{1}{\sum_{i=0}^{k-1}q^{i}}$.
\end{proposition}
\begin{proof}
Since $H$ is a $d$-uniform $d$-regular hypergraph with $d=\frac{q^k-1}{q-1}$,  the constant functions $f(e)=\frac{1}{d}$ for every $e\in E$  and  $g(v)=\frac{1}{d}$ for every $v\in V$ form a fractional matching and a fractional cover for $H$, respectively, of size $q+\frac{1}{\sum_{i=0}^{k-1}q^{i}}$ each.
By linear programming duality, the proposition follows. 
\end{proof}

Since $H$ is a $d$-uniform hypergraph for $d=\frac{q^k-1}{q-1} =\sum_{i=0}^{k-1}q^{i}$, it can be realized as a family $F$ of $d$-intervals (where each interval is just a point in $\mathbb{R}$) with $\tau^*(F)= q+\frac{1}{\sum_{i=0}^{k-1}q^{i}} = \Omega(d^{\frac{1}{k-1}})$.
\qed

\section{The $(p,q)$-property in families of homogeneous $d$-intervals}\label{sec:dpq}

\begin{theorem}\label{dpq*}
If a family $H$ of homogeneous $d$-intervals satisfies the $(p,q)$-property for some fixed integer constants $p\ge q>1$, then $$\tau^*(H)\le \max\Big\{\frac{2^{\frac{1}{q-1}}(ep)^\frac{q}{q-1}}{q}\cdot d^{\frac{1}{q-1}} + 1,~ 2p^2\Big\}.$$
\end{theorem}
\begin{proof}
By removing and duplicating edges in $H$ we may assume that $H$ is a multiset $\{h_1,\dots,h_n\}$ in which each $h_i$ has multiplicity $a_i$, with $m = \sum_{i=1}^n a_i$, and that $\tau^*(H) = \nu^*(H)= m/r$, where $r$ is the maximal number of edges in $H$ with a non-empty intersection. Then clearly $r\ge a_i$ for every $1\le i \le n$. Moreover, since a subset of $p$ elements in $H$ in which no two elements are copies of the same edge must contain some $q$ elements that intersect, we have also $r\ge q$. 

For $1\le i \le n$ and $1\le j \le a_i$ Denote by $h_{i,j}$ the $j-$th copy of $h_i$ in $H$.
Let $T$ be the family of all subsets of cardinality $p$ of $H$ of the form $t=\{h_{i_1,j_1},\dots,h_{i_p,j_p}\}$, with $i_u\neq i_v$ for all $1\le u < v \le p$. Since $r\ge a_i$ we have $$|T| \ge \frac{1}{p!}m \big(m-r \big) \big(m-2r \big)\cdots \big(m-(p-1)r\big) \ge \frac{1}{p!}m^p\big(1-\frac{pr}{m}\big)^p.$$

The $(p,q)$-property implies that for every $t\in T$ there exists an intersecting subset $s\subset t$ of cardinality $q$. Moreover, $s$ is contained in at most $\binom{m-q}{p-q}$ members of $T$. It follows that the number of intersecting subsets of size $q$ in $H$ is at least
$$\frac{|T|}{\binom{m-q}{p-q}} \ge \frac{\binom{m}{q}}{\binom{p}{q}}\big(1-\frac{pr}{m}\big)^p.$$

By Lemma \ref{firstlemma}, each such intersecting subset admits at least two pairs $(x,\{f_1,\dots,f_{q-1}\})$, where $x$ is
an endpoint of an edge $f$ of $H$, $f_1,\dots,f_{q-1}$ are $q-1$ distinct edges of $H$, all different from $f$, and $x$ lies in every $f_i$, $1\le i\le q-1$. Since there are altogether $2dm$ endpoints of members in $H$, there must exist an endpoint $x$ that pierces at least $$X:=\frac{\binom{m}{q}}{dm\binom{p}{q}}\big(1-\frac{pr}{m}\big)^p$$ subsets of cardinality $q-1$ of $H$.
If $\frac{m}{r} \ge 2p^2$, this implies that $x$ lies in at least
\begin{equation}\label{equation}
\begin{split}
\Big((q-1)!\cdot X\Big)^{\frac{1}{q-1}} & \ge
 \Big(\frac{(q-1)!\binom{m}{q}}{dm\binom{p}{q}}\big(1-\frac{pr}{m}\big)^p\Big)^{\frac{1}{q-1}} \\
& \ge \frac{m-q+1}{(dq\binom{p}{q})^{\frac{1}{q-1}}}\Big(1-\frac{p^2r}{m}\Big)^{\frac{1}{q-1}} \\
& \ge \frac{m-q+1}{(2dq\binom{p}{q})^{\frac{1}{q-1}}} \ge \frac{q(m-q+1)}{(2e^qp^qd)^{\frac{1}{q-1}}}
\end{split}
\end{equation}
different edges in $H$.
It follows that $$\frac{q(m-q+1)}{2^{\frac{1}{q-1}}(ep)^\frac{q}{q-1}d^{\frac{1}{q-1}}}\le r,$$ which together with $r\ge q$ entails
$$\tau^*(H)=\frac{m}{r} < \frac{2^{\frac{1}{q-1}}(ep)^\frac{q}{q-1}}{q}\cdot d^{\frac{1}{q-1}} + 1,$$ as we wanted to show.
\end{proof}

The proof of Theorem \ref{dpq} now follows by combining the bound in Theorem \ref{dpq*} with the bound $\tau(H) \le d\tau^*(H)$ given in Theorem \ref{alon}.

\section{A lemma on collections of subtrees}\label{sec:lemma}

\begin{lemma}\label{lemma}
Let $H$ be a finite family of $n$ (not necessarily distinct) subtrees of a tree $G$ and suppose that $\binom{H}{p}$ contains at least $k$ intersecting subsets. Then there exists a vertex in $G$ that lies in at least $$\Big((p-1)!~\frac{k}{n}\Big)^{\frac{1}{p-1}}+1$$ members of $H$.
\end{lemma}
\begin{proof}
Let $H'$ be a sub-family of $H$ with the property that every element in $H'$ belongs to at least $k/n$ different intersecting subsets of size $p$ in $H'$. Such a sub-family exists since we can greedily remove from $H$ elements belonging to less than $k/n$ different intersecting subsets, and since in each step less than $k/n$ intersecting subsets of $p$ elements are removed, this process cannot terminate with an empty set.

Choose an arbitrary vertex $u\in V(G)$ and consider $G$ as a tree rooted in $u$. Let $dist(u,v)$ denote the length of the path from $u$ to a vertex $v\in V(G)$. Then every tree $h\in H'$ is rooted at a vertex $v_h\in V(h)$ with the property that $dist(u,v_h) = \min\{dist(u,x) \mid x\in V(h)\}$.
Let
$h\in H'$ be such that $dist(u,v_h) = \max\{dist(u,v_f) \mid f\in H'\}$. Then if $h$ belongs to an intersecting subset $T$ of size $p$ of $H'$, then every one of the $p-1$ members of $T$ that are different from $h$ must contain $v_h$. Since $h$ belongs to at least $k/n$ different such subsets $T$, we conclude that $v_h$ belongs to at least $k/n$ different subsets of $p-1$ elements in $H'$, all of which do not contain $h$. It follows that $v_h$ belongs to at least $$\Big((p-1)!~ \frac{k}{n}\Big)^{\frac{1}{p-1}}$$ members of $H$ all different from $h$, which proves the lemma.
\end{proof}

\section{The $(p,p)$ and $(p,q)$ properties in families of $d$-trees}\label{sec:dtree}

As before, we first prove bounds on the fractional piercing numbers. 
\begin{theorem}\label{mainpp}
Let $G$ be a tree, and let $H=H(G)$ be a hypergraph of $d$-trees satisfying the $(p,p)$ property for some fixed integer constant $p>1$, then $$\tau^*(H)<(pd)^{\frac{1}{p-1}} + 1.$$
\end{theorem}
\begin{proof}
Let $f:H \to \mathbb{Q}^+$ be a maximal fractional matching. As in the proof of Theorem \ref{dpp}, by removing 
and duplicating edges if necessary,
one may assume without loss of generality that $f(h)=\frac{1}{r}$ for all $h \in H$, where $r$ is the maximal size of an intersecting subset of $H$. In particular, $r\ge p$. Letting $n = |H|$ we have by LP duality $\tau^*(H)=\nu^*(H)=\frac{n}{r}$, and thus our aim is to show that
$\frac{n}{r} < (pd)^{\frac{1}{p-1}} + 1$.

By the $(p,p)$ property, $H$ contains $\binom{n}{p}$ subsets of $p$ elements, each of which intersect at a vertex.
Since each subgraph in $H$ contains at most $d$ connected components, the multiset $H'$ of all subtrees that appear as a connected component in an element of $H$ is of size at most $nd$, and it contains at least $\binom{n}{p}$ intersecting subsets of $p$ elements each.

Applying Lemma \ref{lemma} to $H'$, we conclude that there exists a vertex in $G$ that lies in at least $$\Big((p-1)!~\frac{\binom{n}{p}}{nd}\Big)^{\frac{1}{p-1}}+1 > \frac{n-p+1}{(dp)^{\frac{1}{p-1}}}
$$ members of $H$, implying
$$\frac{n-p+1}{(dp)^{\frac{1}{p-1}}} \le r.$$
Now, as in the proof of Theorem \ref{dpp}, it follows that $$\frac{n}{r}< (dp)^{\frac{1}{p-1}} + 1. $$ \end{proof}

Combining Theorems \ref{mainpp} and \ref{alon}, the proof of Theorem \ref{treedpp}
follows.
Similarly, the proof of Theorem \ref{treedpq} follows by combining  Theorem \ref{alon} with the following: 

\begin{theorem}\label{mainpq*}
Let $G$ be a tree, and let $H=H(G)$ be a hypergraph of $d$-trees satisfying the $(p,q)$ property for some fixed integer constants $p\ge q>1$, then $$\tau^*(H)\le \max\Big\{\frac{2^{\frac{1}{q-1}}(ep)^\frac{q}{q-1}}{q}\cdot d^{\frac{1}{q-1}} + 1,~ 2p^2\Big\}.$$
\end{theorem}

The proof of Theorem \ref{mainpq*} is very similar to that of Theorem \ref{dpq*} (with Lemma \ref{lemma} playing the role of Lemma \ref{firstlemma} there), and therefore is omitted.    
We conclude this section by remarking that the constants in Theorems \ref{dpq} and \ref{treedpq} can be improved by taking more careful estimations, but we avoid from doing so for the sake of simplicity.

\section{Graphs of bounded tree-width}\label{sec:dtw}

A tree decomposition of a graph $G=(V,E)$ is a tree $T$ with vertices $X_1,\dots,X_n$, where each $X_i$ is a subset of $V$, satisfying the following properties:
\begin{itemize}
\item[(i)] $\bigcup_{i=1}^n X_i = V$,
\item[(ii)] If $v\in X_i \cap X_j$, then $v\in X_k$ for all vertices $X_k$ in the path connecting $X_i$ and $X_j$ in $T$.
\item[(iii)] For every edge $uv \in E$, there exists $1\le i \le n$ so that $\{u,v\}\subset X_i$.
\end{itemize}
The width of a tree decomposition is $\max_i\{|X_i|\}-1$. The tree-width of a graph $G$ is the minimum width among all possible tree decompositions of $G$. 
Families of graphs with bounded tree-width include the cactus graphs, pseudoforests, series-parallel graphs, outerplanar graphs, Halin graphs and more.
\medskip

{\em Proof of Theorem \ref{tw}.}
Let $T$ be a tree decomposition of $G$ with vertex set $V(T)=\{X_1,\dots,X_n\}$, where $X_i\subset V$ and $|X_i| \le k+1$ for all $1\le i\le n$. For each subgraph $h\in H$ let $h'$ be the subgraph of $T$ induced on all vertices $X_i$ for which $X_i$ contains a vertex of $h$. Let $H' = \{h'\mid h\in H\}$. Then $H'$ satisfies the $(p,q)$ property and each member of $H'$ has at most $d$ connected components. By Theorem \ref{treedpq} there exists a set $C \subset V(T)$ of cardinality at most $$\max\Big\{\frac{2^{\frac{1}{q-1}}(ep)^\frac{q}{q-1}}{q}\cdot d^{\frac{q}{q-1}} + d,~ 2p^2d\Big\}$$
 that forms a cover of $H'$. Then the set $\bigcup C$ is a cover of $H$, and its cardinality is at most $$(k+1)\cdot\max\Big\{\frac{2^{\frac{1}{q-1}}(ep)^\frac{q}{q-1}}{q}\cdot d^{\frac{q}{q-1}} + d,~ 2p^2d\Big\}.$$

\section*{Acknowledgment} The author is grateful to the members of the LOG(M) project at the University of Michigan, and especially to Yiwang Chen, for helpful discussions regarding the proof of Theorem   \ref{dpps}.


\end{document}